\documentclass[reqno]{amsart}

\usepackage[letterpaper,left=1in,right=1in,top=1in,bottom=1in]{geometry}

\newcommand{\myauthor}{Elden Elmanto and H\aa kon Kolderup}
\newcommand{\mytitle}{On Modules over Motivic Ring Spectra}

\title{On Modules over Motivic Ring Spectra}
\author{\myauthor}
\date{}

\usepackage[pdfstartview=FitH,
            pdfauthor={\myauthor},
            pdftitle={\mytitle},
            colorlinks,
            linkcolor=reference,
            citecolor=citation,
            urlcolor=e-mail,
            backref]{hyperref}

\usepackage{amsmath}

\usepackage{amscd}
\usepackage{amsbsy}
\usepackage{amssymb}
\usepackage{verbatim}
\usepackage{eufrak}
\usepackage{microtype}
\usepackage{hyperref}
\usepackage{mathrsfs}
\usepackage{amsthm}
\usepackage[all,cmtip]{xy}
\usepackage{tikz}
\usetikzlibrary{matrix,arrows}
\usepackage[abbrev,lite]{amsrefs}
\usepackage{mathpazo}
\usepackage{dsfont}

\newcommand{\rig}{\mathsf{rig}}

\usepackage{color}
\definecolor{todo}{rgb}{1,0,0}
\definecolor{conditional}{rgb}{0,1,0}
\definecolor{e-mail}{rgb}{0,.40,.80}
\definecolor{reference}{rgb}{.20,.60,.22}
\definecolor{mrnumber}{rgb}{.80,.40,0}
\definecolor{citation}{rgb}{0,.40,.80}

\let\oldmarginpar\marginpar
\renewcommand\marginpar[1]{\-\oldmarginpar[\raggedleft\footnotesize #1]%
{\raggedright\footnotesize #1}}

\newcommand{\Ascr}{\mathscr{A}}

\newcommand{\Oscr}{\mathscr{O}}

\newcommand{\C}{\mathrm{C}}
\newcommand{\D}{\mathsf{D}}
\newcommand{\F}{\mathrm{F}}
\renewcommand{\H}{\mathrm{H}}
\newcommand{\SH}{\mathrm{SH}}

\newcommand{\MW}{\mathrm{MW}}

\newcommand{\Shv}{\mathrm{Shv}}

\newcommand{\LL}{\mathrm{L}}

\newcommand{\RMod}{\mathrm{LMod}}
\newcommand{\LMod}{\mathrm{LMod}}

\renewcommand{\AA}{\mathbf{A}}

\newcommand{\GG}{\mathbf{G}}

\newcommand{\ZZ}{\mathbf{Z}}

\renewcommand{\P}{\mathrm{P}}

\newcommand{\enh}{\mathsf{enh}}

\newcommand{\op}{\mathrm{op}}

\DeclareMathOperator{\id}{id}

\newcommand{\Free}{\mathsf{Free}}

\newcommand{\Nis}{\mathrm{Nis}}

\newcommand{\Spc}{\mathsf{Spc}}
\newcommand{\CMon}{\mathrm{CMon}}

\newcommand{\Sm}{\mathsf{Sm}}

\newcommand{\eff}{\mathrm{eff}}

\DeclareMathOperator{\Hom}{Hom}
\DeclareMathOperator{\Maps}{Maps}

\newcommand{\Mod}{\mathrm{Mod}}
\renewcommand{\Pr}{\mathrm{Pr}}

\newcommand{\G}{\mathrm{G}}

\newcommand{\Cat}{\mathrm{Cat}}

\newcommand{\stab}{\mathsf{stab}}

\newcommand{\MGL}{\mathsf{MGL}}

\newcommand{\E}{\mathsf{E}}

\DeclareMathOperator*{\colim}{colim}

\newcommand{\mot}{\mathrm{mot}}

\DeclareMathOperator{\DM}{DM}

\DeclareMathOperator{\Sch}{Sch}

\newcommand{\Cor}{\mathsf{Corr}}

\newcommand{\sspt}{{\mathbf 1}}
\newcommand{\s}{{\mathbf S}}

\newcommand{\scrS}{\mathscr{S}}
\newcommand{\scrM}{\mathscr{M}}
\newcommand{\scrP}{\mathscr{P}}
\newcommand{\h}{\sf{h}\kern-0.13pc}
\numberwithin{equation}{section}
\newcommand{\Reg}{\mathsf{Reg}}
\newcommand{\Loc}{\mathsf{Loc}}

\theoremstyle{plain}
\newtheorem{theorem}{Theorem}[section]
\newtheorem*{theorem*}{Theorem}
\newtheorem{lemma}[theorem]{Lemma}

\newtheorem{proposition}[theorem]{Proposition}

\newtheorem{corollary}[theorem]{Corollary}

\newtheoremstyle{named}{}{}{\itshape}{}{\bfseries}{.}{.5em}{#1 \thmnote{#3}}
\theoremstyle{named}

\theoremstyle{definition}
\newtheorem{definition}[theorem]{Definition}

\newtheorem{example}[theorem]{Example}

\theoremstyle{remark}
\newtheorem{remark}[theorem]{Remark}

\begin{document}

\begin{abstract}
We provide an axiomatic framework that characterizes the stable $\infty$-categories that are module categories over a motivic spectrum. This is done by invoking Lurie's $\infty$-categorical version of the Barr--Beck theorem. As an application, this gives an alternative approach to R\"ondigs and \O stv\ae r's theorem relating Voevodsky's motives with modules over motivic cohomology, and to Garkusha's extension of R\"ondigs and \O stv\ae r's result to general correspondence categories, including the category of Milnor--Witt correspondences in the sense of Calm\`es and Fasel. We also extend these comparison results to regular Noetherian schemes over a field (after inverting the residue characteristic), following the methods of Cisinski and D\'eglise.
\end{abstract}
\maketitle

\section{Introduction}

In \cite{ropreprint} and \cite{ro}, R\"ondigs and \O stv\ae r employed the technology of motivic functors developed in \cite{dro} to show an important structural result regarding motivic cohomology, namely that there is an equivalence of model categories between motives and modules over motivic cohomology, at least over fields of characteristic zero. In particular, this implies that Voevodsky's triangulated categories of motives, introduced in \cite{voevodsky-trimot}, is equivalent to the homotopy category of modules over the motivic Eilenberg--Maclane spectra. This result has been extended to bases which are regular schemes over a field in the work of Cisinski--D\'eglise on integral mixed motives in the equicharacteristic case \cite{equichar}. 

These theorems provide pleasant reinterpretations of Voevodsky's category of motives as modules over a highly structured ring spectrum. The analog in topology is the result that chain complexes over a ring $R$ is equivalent (in an appropriate model categorical sense) to modules over the Eilenberg--Mac Lane spectrum $\H R$. This result was first obtained by Schwede and Shipley in \cite{schwede-shipley} as part of the characterization of stable model categories in \emph{loc. cit.}\footnote{We remark that an $\infty$-categorical treatement of the Schwede--Shipley results can be found in \cite{higheralgebra}*{Theorem 7.1.2.1}.}. More recently, R\"ondigs--\O stv\ae r's result was extended to general categories of correspondences by Garkusha in \cite{garkusha}. 

In the present paper, we aim to provide a robust and conceptually simpler approach to the above results. More precisely, by making use of Lurie's $\infty$-categorical version of the Barr--Beck theorem, we derive a characterization of those stable $\infty$-categories that are equivalent to a module category over a motivic spectrum. These categories are examples of  \emph{motivic module categories} in the sense of Definition~\ref{def:mmc}. Examples include $\DM(k)$ in the sense of Voevodsky \cite{mvw} or $\widetilde{\DM}(k)$ in the sense D\'eglise-Fasel \cite{deglise-fasel}. Our characterization then reads as follows:

\begin{theorem}[See Theorem~\ref{thm:fieldcase}] \label{thm:fieldcase-intro}
Let $k$ be a field. If $\scrM(k)$ is a motivic module category on $k$, then there is an equivalence of presentably symmetric monoidal stable $\infty$-categories
\[
\scrM(k)\left[\tfrac1e\right]\simeq\Mod_{R_{\scrM}\left[\tfrac1e\right]}(\SH(k)).
\]
Here, $R_{\scrM}$ is a motivic $\mathcal{E}_{\infty}$-ring in $\SH(k)$ corresponding to the monoidal unit in $\scrM(k)$. In particular, the associated triangulated categories are equivalent.
\end{theorem}

In fact, we formulate a parametrized version of motivic module categories and, under further hypotheses, we show that Theorem~\ref{thm:fieldcase-intro} extends to regular schemes over fields (see Theorem~\ref{thm:reg-mod}). This proof is essentially borrowed from \cite{equichar}.

 The proof method breaks down into three conceptually simple steps:

\begin{enumerate}
\item Prove that a motivic module category $\scrM(k)$ on $k$ 
is equivalent to the category of modules over some monad on $\SH(k)$.
\item Produce a functor from modules over the monad to modules over a corresponding motivic spectrum (Lemma~\ref{lem:barrbeck2}).
\item Determine when this functor is an equivalence. 
\end{enumerate}

We also give a way to engineer many examples of motivic module categories via the notion of \emph{correspondence categories}, on which one can apply the usual constructions of motivic homotopy theory. We hope that, given the proliferation of these maneuvers, streamlining and axiomatizing these constructions can be useful reference for the community.

\subsection{Overview} Here is an outline of this paper:
\begin{itemize}
\item In Section~\ref{sec:prelims} we collect some background material on the Barr--Beck--Lurie theorem, compact rigid generation in motivic homotopy theory and premotivic categories. In Section~\ref{sec:module-cats} we provide an axiomatic framework characterizing the stable $\infty$-categories that are module categories over motivic spectra. 
\item In Section~\ref{sec:module-cats}, we move on to discuss examples of categories satisfying these axioms.
\item The most prominent example are those arising from some sort of correspondences; we make this precise in Section~\ref{sec:corr}. 
\item Finally, in Section~\ref{sec:regk} we prove that the axioms of Section~\ref{sec:module-cats} are satisfied for the correspondence categories of Section~\ref{sec:corr} in various situations.
\end{itemize}

\subsection{Conventions and notation}
We will rely on the language of $\infty$-categories following Lurie's books \cite{htt} and \cite{higheralgebra}. By a \emph{base scheme} we mean a Noetherian scheme $S$ of finite dimension. We denote by $\Sch$ the category of Noetherian schemes, and by $\Sm_S$ the category of smooth schemes of finite type over $S$. We denote by $\mathbb{T}$ the Thom space of the trivial vector bundle of rank $1$ over the base $S$ so that we have the standard motivic equivalences: $\mathbb{T} \simeq \AA^1/\AA^1 \setminus 0 \simeq \mathbb{P}^1.$ We set $\s^{p,q} := (S^1)^{\otimes(p-q)} \otimes \GG_m^{\otimes q}$ and $\Sigma^{p,q} M := \s^{p,q} \otimes M$, suitably interpreted in the category of motivic spaces or spectra. We reserve $\sspt$ for the motivic sphere spectrum in $\SH(k)$ and write $\Sigma^{p,q}\sspt$ for the $(p,q)$-suspension of $\sspt$. If $\tau$ is a topology on $\Sm_S$, we write $\H_{\tau}(S)$ (resp. $\SH_{\tau}(S)$) for the unstable (resp. the $\mathbb{T}$-stable) motivic homotopy $\infty$-category. If $\tau = \Nis$ we may drop the decoration.

\subsection{Acknowledgements} We would like to thank Paul Arne \O stv\ae r for suggesting to us the problem, and Shane Kelly for useful comments and suggestions. We would especially like to thank Tom Bachmann for very useful comments that changed the scope of this paper. Elmanto would like to thank John Francis for teaching him about ``Barr--Beck thinking," Marc Hoyois for suggesting to him this alternative strategy to deriving  \cite{ro} a long time ago, and Maria Yakerson for teaching him about $\MW$-motives. Kolderup would like to thank Jean Fasel and Paul Arne \O stv\ae r for their patience and for always being available for questions.

\section{Preliminaries}\label{sec:prelims}

\subsection{The Barr--Beck--Lurie Theorem}\label{adj1} Let us start out by recalling the Barr--Beck--Lurie theorem characterizing modules over a monad, in the setting of $\infty$-categories. Let $F: \C \rightleftarrows \D: G$ be an adjunction. We use the terminology of \cite{ga-roz1}*{\S 3.7}. Then the endofunctor $GF\colon  \C \rightarrow \C$ is a monad, and the functor $G\colon  \D \rightarrow \C$ factors as: 
\[\D \xrightarrow{\G^{\enh}}\RMod_{G F}(\C) \xrightarrow{u}\C,\]
where $u$ is the forgetful functor. Moreover, the functor $\G^{\enh}\colon  \D \rightarrow \RMod_{GF}(\C)$ admits a left adjoint:
\[\F^{\enh}\colon  \RMod_{G F}(\C) \rightarrow \D.\]

\subsubsection{} \label{adj2} The net result is that the adjunction $F: \C \rightleftarrows \D: G$ factors as

$$\xymatrix{
\C  \ar@/^1.2pc/[rr]^-{F}   \ar@/_1.2pc/[dr]_-{\Free_{GF}}&  & \D \ar[ll]_{G} \ar[dl]_{G^{\enh}}  \\
 & \RMod_{G F}(\C) \ar[ul]_{u} \ar@/_1.2pc/[ur]_-{F^{\enh}}. &
}$$
Here, the functor $\Free_{GF}\colon  \C \rightarrow \RMod_{GF}(C)$ is simply the left adjoint to the functor $u$ appearing in the factorization of $G$ above, and thus deserves to be called the ``free $GF$-module" functor.

\subsubsection{}The Barr--Beck--Lurie theorem provides necessary and sufficient conditions for the functor $G^\enh\colon \D\to\LMod_{GF}(\C)$ to be an equivalence. Before stating the theorem, recall first that a simplicial object $X_{\bullet}\colon \Delta^{\op} \rightarrow \D$ is \emph{split} if it extends to a split augmented object; in other words it extends to a functor $U\colon\Delta^{\op}_{-\infty} \rightarrow \D$. Here $\Delta_{-\infty}$ is the category whose objects are integers $\geq -1$, and where $\Hom_{\Delta_{-\infty}}(n, m)$ consists of nondecreasing maps $n \cup \{-\infty\} \rightarrow m \cup \{-\infty\}.$ Every split augmented simplicial diagram is a colimit diagram so that the map $\colim X_{\bullet} \rightarrow X_{-1}$ is an equivalence. If $G\colon \D \rightarrow \C$ is a functor, we say that a simplicial object $X_{\bullet}$ in $\D$ is \emph{$G$-split} if $G \circ X_{\bullet}$ is split.

\begin{theorem} [Barr--Beck--Lurie \cite{higheralgebra}*{Theorem 4.7.3.5}] \label{thm:barrbeck1} Let $G\colon  \D \rightarrow \C$ be a functor of $\infty$-categories admitting a left adjoint $F\colon  \C \rightarrow \D$. Then the following are equivalent:

\begin{enumerate}
\item The functor $\G^{\enh}$ and $\F^{\enh}$ are mutually inverse equivalences.
\item The functor $\G^{\enh}$ is conservative, and for any simplicial object $X_{\bullet}\colon\Delta^\op\rightarrow \D$ which is $G$-split, $X_{\bullet}$ admits a colimit in $\D$. Furthermore, any extension $\overline{X_{\bullet}}\colon  (\Delta^{\op})^{\vartriangleright} \rightarrow \D$ is a colimit diagram if and only if $G \circ \overline{X_{\bullet}}$ is.
\end{enumerate}
\end{theorem}

Any adjunction $(F,G)$ satisfying the equivalent conditions above is called \emph{monadic}.

\subsection{Compact and rigid objects in motivic homotopy theory}\label{sect:cpct-gen} 
We now recall some facts about compact-rigid generation in motivic stable $\infty$-categories.

\subsubsection{} For now we work over an arbitrary base $S$. Denote by:
\begin{enumerate}
\item $\SH^{\omega}(S)$ the full subcategory of $\SH(S)$ spanned by the compact objects, and
\item $\SH^{\rig}(S)$ the full subcategory of $\SH(S)$ spanned by the strongly dualizable objects.
\end{enumerate}
Then, as proved in \cite{riou-sw}, there is an inclusion
\begin{equation} \label{riou-inclusions}
\SH^{\rig}(S) 
\subseteq 
\SH^{\omega}(S).
\end{equation}
Indeed, the argument in \emph{loc cit.} only relies on the computation of the Spanier--Whitehead dual of the suspension spectrum of a smooth projective $S$-scheme as the Thom spectrum of its stable normal bundle; this has been established generally by Ayoub in \cite{ayoub}.

\subsubsection{} The $\infty$-category $\SH(S)$ is generated under sifted colimits by $\Sigma^{q}_{\mathbb{T}} \Sigma^{\infty}_{\mathbb{T}}X_+$ where $X$ is an affine smooth scheme over $S$ and $q \in \ZZ$ \cite{adeel}*{Proposition 4.2.4}. Furthermore, each generator is a compact object in $\SH(S)$ since Nisnevich sheafification preserves filtered colimits (see, for example, \cite{hoyois-sixops}*{Proposition 6.4} where we set the group of equivariance to be trivial). Hence the $\infty$-category $\SH^{\omega}(S)$ is generated under finite colimits and retracts by $\Sigma_{\mathbb{T}}^{q}\Sigma^{\infty}_{\mathbb{T}}X_+, q \in \ZZ$, where $X$ is affine. From this discussion it is immediate that:

\begin{lemma} \label{lem:collapse} Let $S$ be a base scheme and let $L\colon  \SH(S) \rightarrow \SH(S)$ be a localization endofunctor. The following are equivalent:
\begin{enumerate}
\item For any smooth affine $S$-scheme $X$, the suspension spectrum $L(\Sigma^{\infty}_{\mathbb{T}}X_+)$ is a retract of some $L(\Sigma^{\infty}_{\mathbb{T}}Y_+)$ where $Y$ is a smooth projective $S$-scheme.
\item The inclusion~\eqref{riou-inclusions} collapse an equality
$
L(\SH^{\rig}(S))
=
L(\SH^{\omega}(S)).
$
 \end{enumerate}
\end{lemma}

\begin{example} \label{exmp-field} Let $k$ be a field and suppose that $\ell$ is a prime which is coprime to the exponential characteristic $e$ of $k$. Let $L_{(\ell)}\colon  \SH(k) \rightarrow \SH(k)$ be the localization endofunctor at $\ell$. If $k$ is perfect then, according to  \cite{yang-zhao-levine}*{Corollary B.2}, condition $1$ of Lemma~\ref{lem:collapse} is satisfied so that 
$
\SH^{\rig}(k)_{(\ell)} 
=
\SH^{\omega}(k)_{(\ell)}$. This immediately implies the equality
$
\SH^{\rig}(k) \left[\tfrac{1}{e}\right]
=
\SH^{\omega}(k)\left[\tfrac{1}{e}\right].
$
We note that this equality is extended to the case of arbitrary fields in \cite{perfect}. More precisely, \cite{perfect} shows that the $\infty$-categories $\SH(k_{\mathrm{perf}})\left[\frac{1}{e}\right]$ and $\SH(k)\left[\frac{1}{e}\right]$ are equivalent. 
\end{example}
\subsubsection{} For a general base scheme $S$, it is not known if the inclusions in~\eqref{riou-inclusions} collapse to equalities even after applying localizations at  a prime $\ell$ or inversions. Hence it is useful to record when it does:
\begin{definition} Let $L\colon  \SH(S) \rightarrow \SH(S)$ be a localization endofunctor. We say that $L(\SH(S))$, or simply $L$, has \emph{compact-rigid generation} if any of the equivalent conditions of Lemma~\ref{lem:collapse} are satisfied.
\end{definition}
Hence Example~\ref{exmp-field} tells us that $\SH(k)_{(\ell)}$ and $\SH(k)\left[\frac{1}{e}\right]$ have compact-rigid generation.

\subsection{Premotivic categories and adjunctions.} \label{premot} Lastly, we recall Cisinski and D\'eglise's notion of a \emph{premotivic category} \cite{CDbook}. Suppose that $\scrS$ is a full subcategory of the category $\Sch$ of Noetherian schemes, and let $\scrP$ denote a class of admissible morphisms \cite[\S 1.0]{CDbook}. In fact, the only example we care about is when $\scrP$ is the class of smooth morphisms. As in \cite{CDbook}*{\S 1} (see \cite{etalemotives}*{Appendix A} for a more succint discussion), a functor 
\[
\scrM\colon\scrS^\op\to\Cat_\infty
\] 
is called a \emph{$\scrP$-premotivic category over $\scrS$} if for each morphism $f\colon  T \rightarrow S$ in $\scrS$, the induced functor $f^*\colon  \scrM(S) \rightarrow \scrM(T)$ admits a right adjoint $f_*$, and if $f$ is admissible, it admits a left adjoint $f_\#$. The left adjoints are furthermore required to satisfy the \emph{$\scrP$-base change formula}, i.e., the exchange morphism $Ex_\#^*\colon q_\# g^*\to f^*p_\#$ is an equivalence whenever
\[
\xymatrix{
Y\ar[r]^{q}\ar[d]_{g} & X\ar[d]^{f}\\
S\ar[r]^{p} & T
}\]
is a Cartesian diagram in $\scrS$ such that $p$ is a $\scrP$-morphism. See \cite[1.1.9]{CDbook} for details.

We refer the reader to the thesis of Khan \cite{adeel} for a detailed discussion of this notion in the $\infty$-categorical setting. If the context is clear, we simply refer to $\scrM$ as a premotivic category. We may also speak of premotivic $\infty$-categories taking values in other (large) $\infty$-categories such as $\Cat_{\infty}^{\otimes}$, $\Cat_{\infty,\stab}$ or $\Pr^L$.

\subsubsection{} \label{premot1}We also have the appropriate notion of an adjunction between premotivic categories (see \cite{CDbook}*{Definition 1.4.6}, \cite{etalemotives}*{Definition A.1.7}). Suppose that $\scrM, \scrM'$ are premotivic categories, then a \emph{premotivic adjunction} is a transformation $\gamma^*\colon  \scrM \rightarrow \scrM'$ such that 
\begin{enumerate}
\item for each $S \in \scrS$, the functor $\gamma_S^*\colon  \scrM(S) \rightarrow \scrM'(S)$ admits a right adjoint $\gamma_{S*}$.
\item For each morphism $f\colon T \rightarrow S \in \scrS$, the canonical transformation $f_{\#}\gamma^*_S \rightarrow \gamma^*_Tf_{\#}$ is an equivalence.
\end{enumerate}
Furthermore we say that a premotivic adjunction $\gamma^*$ is a \emph{localization of premotivic categories} (or, simply, a \emph{localization}) if for each $S \in \scrS$ the functor $\gamma_{S*}$ is fully faithful, i.e., a localization in the sense of \cite{htt}*{Definition 5.2.7.2}. Furthermore we say that a localization of premotivic categories is \emph{smashing} if $\gamma_{S*}$ preserves colimits. Suppose further that $\scrM$ takes values in $\Cat^{\otimes}_{\infty}$. In particular, the functors $f^*$ are strongly symmetric monoidal. Then a localization is \emph{symmetric monoidal} if for any $S \in \scrS$ then if $\E \in \scrM(S)$ is $L$-local, then for any $\F \in \scrM(S)$, $\E \otimes \F$ is $L$-local as well. This last condition implies that the symmetric monoidal structure on $\scrM(S)$ descends to one on the subcategory of $L$-local objects and the localization functor is strongly symmetric monoidal \cite{higheralgebra}*{Proposition 2.2.1.9}.

\subsubsection{} We recall two conditions on $\scrM$ which will be relevant to us later. In order to formulate them, we will now assume that $\scrM$ takes values in stable $\infty$-categories. Let $S\in\scrS$ be a scheme. Suppose that $i\colon Z\to S$ is a closed subscheme, and let $j\colon U\to S$ be its open complement.

\begin{definition} \label{def:loci} 
Let $\scrM\colon\scrS^\op\to\Cat_{\infty,\stab}$ be a premotivic category, and let $Z\xrightarrow{i} S\xleftarrow{j}U$ be as above. We say that $\scrM$ satisfies $(\Loc_i)$ if
\[
\scrM(Z)\xrightarrow{i_*}\scrM(S)\xrightarrow{j^*}\scrM(U)
\]
is a cofiber sequence of stable $\infty$-categories.

We say that $\scrM$ satisfies $(\Loc)$ if $(\Loc_i)$ is satisfied for any closed immersion $i$.
\end{definition}
 
Let $c = (c_i)_{i \in I}$ be a collection of Cartesian sections of $\scrM$ (the only case we consider is $\{\Sigma^{p,q}\sspt \}_{p, q \in \ZZ}$). We denote by $\scrM_c(S) \subseteq \scrM(S)$ the smallest thick subcategory of $\scrM(S)$ which contains $f_{\#}f^*c_{i,X}$ for any smooth morphism $f\colon  T \rightarrow S$. Following \cite{equichar}*{Definition 2.3}, we call objects in $\scrM_c(S)$ \emph{c-constructible}. We say that $\scrM$ is \emph{c-generated} if for all $X \in \scrS$ the stable $\infty$-category $\scrM(S)$ is generated by $\scrM_c(S)$ under all small colimits.

\begin{definition} \label{def:cont-loc} Let $\scrM\colon\scrS^\op\to\Cat_{\infty,\stab}$ be a premotivic category. Suppose that $\Ascr \subseteq \scrS^{\Delta^1}$ is a collection of morphisms in $\scrS$. We say that $\scrM$ is \emph{continuous with respect to $\Ascr$} if the following holds. Suppose that $X\colon  I \rightarrow \scrS$ is a cofiltered diagram in $\scrS$ whose transition maps belongs to $\Ascr$ and whose limit $X:=\lim_{i\in I} X_i$ exists in $\scrS$. Then the canonical map
\[
\scrM_c(X) \rightarrow \lim_{i\in I} \scrM_c(X_i).
\]
is an equivalence.
\end{definition}

\section{Motivic module categories}\label{sec:module-cats}
In this section we formulate the notion of \emph{motivic module categories} and relate it to categories of modules over a motivic $\mathcal{E}_\infty$-ring spectrum. 

\subsubsection{}Let $\scrS$ be a full subcategory of $\Sch$. By \cite{ayoub,CDbook} we then have a premotivic category $\SH|_\scrS\colon  \scrS \rightarrow \Pr^{L,\otimes}_{\stab}$ whose value at $S\in\scrS$ is the motivic stable homotopy category $\SH(S)$ over $S$.

\begin{definition} \label{def:mmc}
Let $\scrS$ be as above, and suppose that $L\colon \SH|_\scrS\to L(\SH)|_\scrS$ is a localization which is symmetric monoidal in the sense of \S\ref{premot1}. We then define the following:
\begin{enumerate}
\item Let $S \in \scrS$. An \emph{$L$-local motivic module category on $S$} is a presentably  symmetric monoidal stable $\infty$-category $\scrM(S)$ equipped with an adjunction
\[
\gamma_S^*: L(\SH(S)) \rightleftarrows \scrM(S):\gamma_{S*}
\]
such that the left adjoint $\gamma_S^*$ is symmetric monoidal, and the right adjoint $\gamma_{S*}$ is conservative and preserves sifted colimits.
\item  An \emph{$L$-local motivic module category over $\scrS$} (or, simply, a \emph{motivic module category} if the context is clear) is a premotivic category
\[
\scrM\colon\scrS^\op\to\Pr^{L,\otimes}_{\stab}
\]
valued in presentably symmetric monoidal stable $\infty$-categories, along with a premotivic adjunction
\[\gamma^*\colon L(\SH)|_{\scrS}\to\scrM;\quad S \mapsto (\gamma_S^*\colon L(\SH(S)) \rightarrow \scrM(S)),
\] 
which evaluates to an $L$-local motivic module category $\scrM(S)$ on $S$ for each $S \in \scrS$.
\end{enumerate}
If $L$ is the identity functor, then we simply say that $\scrM$ is a \emph{motivic module category}. When the localization $L$ is clear, we may denote a motivic module category by a pair ($\SH|_\scrS,\scrM)$.
Moreover, if the scheme $S$ is implicitly understood, we may drop the $S$ from the notation $(\gamma_S^*,\gamma_{S*})$.
\end{definition}


In \S\ref{sec:corr} we will give a way to construct motivic module categories using very general inputs. 

\begin{lemma}\label{lemma:eoo}
Let $S\in\scrS$, and let $\sspt_S\in\SH(S)$ denote the motivic sphere spectrum over $S$. If $\scrM$ is an $L$-local motivic module category, then the spectrum $L\gamma_*\gamma^*(\sspt_S)\in\SH(S)$ is an $\mathcal{E}_{\infty}$-ring spectrum.
\end{lemma}

\begin{proof}
As $\gamma_*$ is lax symmetric monoidal, it follows that $\gamma_*$ preserves $\mathcal{E}_{\infty}$-algebras. Since $\gamma^*$ is symmetric monoidal, $\gamma^*(\sspt_S)$ is the unit object in $\scrM$ and is thus an $\mathcal{E}_{\infty}$-algebra. As $L$ is symmetric monoidal, we conclude that $\gamma_*\gamma^*(\sspt_S)$ is an $\mathcal{E}_{\infty}$-ring spectrum.
\end{proof}

\subsubsection{}The Barr--Beck--Lurie theorem ensures that a motivic module category on $S$ is always equivalent to modules over a monad, as the following lemma records. We will subsequently investigate when we can further enhance this equivalence to modules over the $\mathcal{E}_{\infty}$-ring spectrum $L\gamma_*\gamma^*(\sspt_S)$.

\begin{lemma}\label{lemma:bbl}
If $\scrM(S)$ is a motivic module category on $S$, then the induced adjunction
\[
\gamma^{*,\enh}:\LMod_{\gamma_*\gamma^*}(L(\SH(S))) \rightleftarrows \scrM(S):\gamma_*^{\enh}
\]
is an equivalence of $\infty$-categories.
\end{lemma}

\begin{proof}
By assumption, the conditions of Theorem~\ref{thm:barrbeck1} are satisfied.
\end{proof}

\subsection{Motivic module categories versus categories of modules} 
The following definition will be essential in relating a motivic module category to a category of modules over a motivic $\mathcal{E}_{\infty}$-ring spectrum.

\begin{definition} \label{def:proj-formula}
Let $\scrM$ be an $L$-local motivic module category over $\scrS$ and let $S \in \scrS$. We say that the pair ($\SH|_\scrS,\scrM)$ \emph{admits the projection formula at $S$} if the transformation
\[
\gamma_{*}\gamma^*(\sspt_S)\otimes(-)\to \gamma_{*}\gamma^*
\]
is an equivalence in $L(\SH(S))$. If ($\SH|_\scrS,\scrM)$ admits the projection formula at any $S\in\scrS$, we say that ($\SH|_\scrS,\scrM)$ \emph{admits the projection formula}.
\end{definition}

\begin{theorem}\label{thm:main}
Let $\scrM$ be an $L$-local motivic module category over $\scrS$. Suppose that $S \in \scrS$ is a scheme such that $(\SH|_\scrS,\scrM)$ admits the projection formula at $S$. Then there is an equivalence of presentably symmetric monoidal stable $\infty$-categories
\[
\scrM(S)\simeq\Mod_{L\gamma_*\gamma^*(\sspt_S)}(\SH(S)).
\]
Consequently, if $(\SH|_\scrS,\scrM)$ admits the projection formula, then we have an equivalence of premotivic categories 
\[\scrM\simeq \Mod_{L\gamma_*\gamma^*(\sspt)}(\SH(-)).\]
\end{theorem}

\subsubsection{} In light of Lemma~\ref{lemma:bbl}, we can prove Theorem~\ref{thm:main} by means of relating modules over monad $\gamma_*\gamma^*$ with modules over the motivic spectrum $\gamma_*\gamma^*(\sspt_S)$. Thus, given $S\in\scrS$ our task is to formulate a relationship between the two $\infty$-categories
\[
\LMod_{\gamma_*\gamma^*}(\SH(S))\quad \text{ and }\quad \LMod_{\gamma_*\gamma^*(\sspt_S) \otimes (-)}(\SH(S)).
\]
To do so, it suffices produce a map of monads 
\[
c\colon \gamma_*\gamma^*(\sspt_S) \otimes (-) \rightarrow \gamma_*\gamma^*,
\] 
which will induce a functor  
\[
c^*\colon  \LMod_{\gamma_*\gamma^*(\sspt_S) \otimes (-)}(\SH(S)) \rightarrow \LMod_{\gamma_*\gamma^*}(\SH(S)).
\]
For this, we appeal to a general lemma.

\begin{lemma} \label{lem:barrbeck2} Let $\C, \D$ be symmetric monoidal $\infty$-categories and suppose that we have an adjunction $F: \C \rightleftarrows \D: G$ such that $F$ is symmetric monoidal (so that $G$ is lax symmetric monoidal). Then there is a map of monads
\begin{equation} \label{c}
c\colon GF(1)\otimes(-)\to GF,
\end{equation}
which gives rise to a commutative diagram of adjunctions
$$\xymatrix{
\C  \ar@/^1.2pc/[rr]^-{\Free_{GF}}   \ar@/_1.2pc/[dr]_-{GF(1) \otimes (-)}&  & \RMod_{GF}(\C) \ar[ll]_{u} \ar[dl]_{c_*}  \\
 & \RMod_{GF(1)\otimes(-)}(\C) \ar[ul]_{u} \ar@/_1.2pc/[ur]_-{c^*}. &
}$$
\end{lemma}

\begin{proof} Since $F$ is monoidal and $G$ is lax monoidal, the functor $GF$ is lax monoidal. Hence $GF(1)$ is an algebra object of $\C$, and thus $GF(1) \otimes (-)$ is indeed a monad. We construct the map of monads $c\colon GF(1) \otimes (-) \rightarrow GF(-)$ by letting $c$ be the composite of the following maps of monads
\begin{align*}
 GF(1)  \otimes (-) & \simeq    (GF(1) \otimes (-) ) \circ \id \\
 & \xrightarrow{\id \circ \epsilon}    (GF(1) \otimes (-) ) \circ GF(-)  \\
 & \xrightarrow{\mu}  G(F(1) \otimes F(-)) \\
 & \simeq GF.
\end{align*}
Here $\epsilon$ is the unit of the adjunction $(F,G)$. The transformation $\epsilon$ is a map of monads via the triangle identities, and the map $\id \circ \epsilon$ is a map of monads since we are $\circ$-tensoring two maps of monads. The map $\mu$ is given by the lax monoidal structure of $G$; more precisely, we note that the endofunctor $G(A \otimes F(-))$ is a monad for any algebra object $A$, and so  $G(F(1) \otimes F(-))$ is in particular a monad. We have a canonical equivalence of monads 
\[(GF(1)  \otimes (-) ) \circ GF(-)  \simeq GF(1) \otimes GF(-).\] 
The lax structure of $G$ then provides a morphism of endofunctors 
\[GF(1) \otimes GF(-) \rightarrow G(F(1) \otimes F(-)) \simeq GF(-),\] 
and the lax structure also verifies that this is a map of monads. This gives rise to a functor $c_*\colon  \RMod_{GF}(\C) \rightarrow \RMod_{GF(1)\otimes (-)}(\C)$, which has a left adjoint by the adjoint functor theorem. 

To obtain the desired factorizations, we note that we have the following commutative diagram of forgetful functors
\[\xymatrix{
\C   &  & \RMod_{GF}(\C) \ar[ll]_{u} \ar[dl]_{c_*}  \\
 & \RMod_{GF(1)\otimes(-)}(\C) \ar[ul]_{u} . &
}\]
Thus the left adjoints also commute.
\end{proof}

\subsubsection{} \label{sect:start} We can now apply Lemma~\ref{lem:barrbeck2} to prove Theorem~\ref{thm:main}.

\begin{proof}[Proof of Theorem~\ref{thm:main}]  
We claim that the adjunction of Lemma~\ref{lem:barrbeck2},
\[
c^*: \LMod_{\gamma_*\gamma^*(\sspt_S)}(\SH(S)) \rightleftarrows \LMod_{\gamma_*\gamma^*}(\SH(S)): c_*,
\] 
is an equivalence.  By the construction in the proof of Lemma~\ref{lem:barrbeck2}, the above adjunction arises from a map of monads given by $c\colon  \gamma_*\gamma^*(\sspt_S) \otimes (-) \rightarrow \gamma_*\gamma^*$. Since $(\SH|_\scrS,\scrM)$ satisfies the projection formula, we conclude that the adjunction $(c^*,c_*)$ is an equivalence.

Now, note that Theorem~\ref{thm:barrbeck1} and Lemma~\ref{lem:barrbeck2} are phrased for $\mathcal{E}_1$-algebras and left modules. However, as $\gamma_*\gamma^*(\sspt)$ is an ${\mathcal E}_\infty$-ring spectrum by Lemma~\ref{lemma:eoo}, the $\infty$-categories of left and right $\gamma_*\gamma^*(\sspt)$-modules are equivalent. We thus conclude that there is an natural equivalence
\[
\Mod_{\gamma_*\gamma^*(\sspt_S)}(\SH(S))\simeq\scrM(S)
\] 
of $\infty$-categories, which carries $\gamma_*\gamma^*(\sspt_S)$ to the unit object $\gamma^*(\sspt_S)$ of $\scrM(S)$. Finally, if $\scrM$ satisfies the projection formula at any $S\in\scrS$, then the naturality of the above equivalence furnishes the equivalence of premotivic categories $\scrM\simeq \Mod_{\gamma_*\gamma^*(\sspt)}(\SH(-))$.
\end{proof}

\begin{remark} In fact, the above reduction can be achieved using a more refined version of Lurie's Barr--Beck theorem \cite{higheralgebra}*{Proposition 4.8.5.8}.\end{remark}

\begin{remark}
We were also informed by Niko Naumann that the above result is a consequence of \cite{mnn}*{Proposition 5.29}.
\end{remark}
In Section~\ref{sec:regk} we will provide examples for which the hypotheses of Theorem~\ref{thm:main} are satisfied.

\section{Correspondence categories} \label{sec:corr}


%
The prime examples of motivic module categories are built from various notions of correspondences. In this section we will give an axiomatization of $\infty$-categories that behave like the category of framed correspondences as in \cite{ehksy}; Suslin--Voevodsky's category of finite correspondences \cite{susvoe}, \cite{mvw}*{Chapters 1 and 2}; Calm\` es and Fasel's Milnor--Witt correpondences \cite{calmes-fasel, deglise-fasel}; Grothendieck--Witt correspondences \cite{druzhinin}; and, more recently, the categories of correspondences studied in \cite{ehksy2} and \cite{dk}. These examples will be discussed in \S\ref{sec:ex}. To begin with, consider the discrete category $\Sch_{S+},$ whose objects are $S$-schemes of the form $X_+$ and morphisms which preserve the base point. We consider the subcategory $\Sm_{S+} \subseteq \Sch_{S+}$ spanned by smooth $S$-schemes of finite type.  We will use heavily the \emph{nonabelian derived $\infty$-category} $\P_{\Sigma}(\C)$ associated to an $\infty$-category $\C$ with finite products; more detailed treatments of this construction can be found in \cite{bachmann-hoyois}*{Chapter 1} and \cite{htt}*{5.5.8}. 

\newcommand{\CorrCat}{\mathrm{CorrCat}}
\newcommand{\fincoprod}{\mathrm{fin.coprod}}
\newcommand{\PreAdd}{\mathrm{PreAdd}}

\begin{definition} \label{def:corr-infty} A \emph{correspondence category} (over a base scheme $S$) is a preadditive $\infty$-category $\C$\footnote{Recall that a preadditive $\infty$-category is one that is pointed, has finite products and coproducts and the map $X \coprod Y \rightarrow X \times Y$ is an equivalence for all $X, Y \in \C$.} equipped with a \emph{graph functor} 
\begin{equation}\label{graph}
\gamma_{\C}\colon \Sm_{S+} \rightarrow \C
\end{equation}
satisfying the following conditions:
\begin{enumerate}
\item The functor $\gamma_{\C}$ is essentially surjective and preserves finite coproducts \footnote{Including the empty coproduct, so that the $\gamma_C$ also preserves the base point of $\Sm_{S+}$.}, so that we get an induced functor
\[
\gamma_*\colon  \P_{\Sigma}(\C) \rightarrow \P(\Sm_S);\quad \mathscr{F} \mapsto \mathscr{F} \circ \gamma_C.
\]
\item The composite functor
\begin{equation}\label{free}
\Sm_{S+} \rightarrow \C \rightarrow \P_{\Sigma}(\C) \stackrel{\gamma_*}{\rightarrow} \P_{\Sigma}(\Sm_{S+}),
\end{equation}
has a right lax $\Sm_{S+}$-linear structure. We abusively denote the composite~\eqref{free} by $\gamma_C(-)$; the context will always make it clear.
\end{enumerate}
The $\infty$-category of correspondence categories $\CorrCat$ is defined as a full subcategory of $\PreAdd_{\infty, \Sm_{S+}/}$, the (large) $\infty$-category of small preadditive $\infty$-categories and functors that preserve finite coproducts equipped with a finite coproduct-preserving functor from $\Sm_{S+}$.\footnote{More succinctly, $\CorrCat$ is the pullback of $\infty$-categories $\PreAdd \times_{\Cat^{\amalg}_{\infty}} \{\Sm_{S+}\}$, where $\Cat^{\amalg}$ denotes $\infty$-categories with finite coproducts and finite coproduct-preserving functors.}
\end{definition}

\subsubsection{} We begin with a couple of clarifying remarks and an example.

\begin{remark} Informally, the $\Sm_{S+}$-linear structure on $\gamma_{\C}(-)$ encodes for any $X, Y \in \Sm_S$ maps
\[
X_+ \otimes \gamma_{\C}(Y_+) \rightarrow \gamma_{\C}(X_+ \otimes Y_+)
\]
in $\P_{\Sigma}(\Sm_{S+})\simeq \P_{\Sigma}(\Sm_S)_*$, which are subject to various compatibilites. For example, if $f\colon  X_+ \rightarrow Z_+$ is a map in $\Sm_{S+}$ then we have a $2$-cell witnessing the commutativity of
\[
\xymatrix{
X_+ \otimes \gamma_{C}(Y_+) \ar[d]_{f \otimes \id} \ar[r] & \gamma_{C}(X_+ \otimes Y_+) \ar[d]^{\gamma_{C}(f \otimes \id)} \\ 
Z_+ \otimes \gamma_{C}(Y_+) \ar[r] & \gamma_{C}(Z_+ \otimes Y_+).
}
\]
Similarly, if $g\colon  Y_+ \rightarrow Z_+$ is a map in $\Sm_{S+}$ then we have a $2$-cell witnessing the commutativity of
\[
\xymatrix{
X_+ \otimes \gamma_{C}(Y_+) \ar[d]_{\id \otimes g} \ar[r] & \gamma_{C}(X_+ \otimes Y_+) \ar[d]^{\gamma_{C}(\id \otimes g)} \\ 
X_+ \otimes \gamma_{C}(Z_+) \ar[r] & \gamma_{C}(X_+ \otimes Z_+).
}
\]
These cells are required satisfy an infinite list of coherences. 
\end{remark}

\begin{remark} \label{rem:smc} The $\Sm_{S+}$-linearity assumption will be satisfied if $\C$ has a symmetric monoidal structure and the functor $\gamma_\C$ is symmetric monoidal. In more detail, we denote by $\CorrCat^{\otimes}$ the $\infty$-category of preadditive $\infty$-categories with a symmetric monoidal structure such that the graph functor $\gamma_\C\colon  \Sm_{S+} \rightarrow \C$ is symmetric monoidal, essentially surjective and preserves finite coproducts. There is a forgetful functor $\CorrCat^{\otimes} \rightarrow \CorrCat$; the second part of Definition~\ref{def:corr-infty} is obtained from the strong symmetric monoidality of $\gamma_C$. This is the case in the examples considered in this paper, but we include it as an axiom to clarify proofs of certain properties. 
\end{remark}

\begin{example} Let $\Cor_S^{\mathrm{clopen}}$ denote the discrete category whose objects are smooth $S$-schemes and morphisms are $X \hookleftarrow Y \rightarrow Z$ such that $X \hookleftarrow Y$ is a summand inclusion. There is an equivalence of categories $\Sm_{S+} \stackrel{\simeq}{\rightarrow} \Cor_S^{\mathrm{clopen}}$ which takes a morphism $f\colon  X_+ \rightarrow Y_+$ to $X \hookleftarrow f^{-1}(Y) \rightarrow Y.$ This graph functor witnesses $\Cor_S^{\mathrm{clopen}}$ as a correspondence category.
\end{example}

\subsubsection{} We now provide some elementary properties of a correspondence category.

\begin{proposition} \label{prop:formal1} Let $\C$ be a preadditive $\infty$-category equipped with an essential surjection 
\[
\gamma_{\C}\colon  \Sm_k \rightarrow \C
\] 
which preserves coproducts, and let $\gamma_{\C*}$ denote the induced functor 
\[
\gamma_{\C*}\colon  \P_{\Sigma}(\C) \rightarrow \P_{\Sigma}(\Sm_S);\quad \mathscr{F} \mapsto \mathscr{F} \circ \gamma_C.
\] 
Then the following properties hold:
\begin{enumerate}
\item The $\infty$-category $\P_{\Sigma}(\C)$ is presentable and preadditive.
\item The functor $\gamma_{C*}$ preserves sifted colimits.
\item The functor $\gamma_{C*}$ is conservative.
\end{enumerate}
\end{proposition}
\begin{proof}
Presentability of $\P_{\Sigma}(\C)$ is \cite{htt}*{Proposition 5.5.8.10.1}, while $\P_{\Sigma}$ applied to a preadditive $\infty$-category is again preadditive by \cite{ggn}*{Corollary 2.4}. The functor $\gamma_{C*}$ preserves sifted colimits since sifted colimits are computed pointwise  (a direct consequence of \cite{htt}*{5.5.8.4.10} parts $4$ and $5$), while $\gamma_{C*}$ is conservative since $\gamma_{\C}$ is essentially surjective.
\end{proof}

\subsubsection{}\label{sect:extend} The composite of $\gamma_\C$ with Yoneda functor $\Sm_{S+} \stackrel{\gamma_\C}{\rightarrow} \C \stackrel{y}{\rightarrow} \P_{\Sigma}(\C)$ has a canonical sifted colimit-preserving extension $\gamma^*_\C\colon  \P_{\Sigma}(\Sm_{S+}) \rightarrow \P_{\Sigma}(C)$. It is easy to check that $\gamma_{\C*}$ is the right adjoint to $\gamma_\C^*$ and thus $\gamma^*_\C$ preserves all small colimits. As a result, we have an adjunction

\begin{equation} \label{gamma-adjunction}
\gamma^*_{C}: \P_{\Sigma}(\Sm_{S+}) \rightleftarrows \P_{\Sigma}(\C): \gamma_{C*}.
\end{equation}
It is also easy promote the $\Sm_{S+}$-linear structure given by the second axiom of a correspondence category to a $\P_{\Sigma}(\Sm_{S+})$-linear structure so that the functor
\[
\gamma_{C*}\circ \gamma_C^*\colon  \P_{\Sigma}(\Sm_{S+}) \rightarrow \P_{\Sigma}(\Sm_{S+})
\]
extends to a right lax $\P_{\Sigma}(\Sm_{S+})$-linear functor.

%
%

\subsubsection{} Now we would like to do motivic homotopy theory on $\C$. Recall that if $X, Y \in \P_{\Sigma}(\Sm_{S+})$, then $X$ is \emph{$\AA^1$-homotopy equivalent to $Y$} if there are maps $f\colon  X \rightarrow Y, g\colon  Y \rightarrow X$ and $\AA^1$-homotopies $H\colon  \AA^1_+ \otimes X \rightarrow X, H'\colon  \AA^1_+ \otimes Y \rightarrow Y$ from $gf$ and $fg$ to the respective identity morphisms. We note that any $\AA^1$-homotopy equivalence is an $\LL_{\AA^1}$-equivalence \cite{morel-voevodsky}*{§2 Lemma 3.6}. 

\begin{lemma} \label{lem:homotopy-pres} The functor $\gamma_{C}\colon  \P_{\Sigma}(\Sm_{S+}) \rightarrow \P_{\Sigma}(\Sm_{S+})$ preserves $\AA^1$-homotopy equivalences.
\end{lemma}

\begin{proof} Suppose that we have a homotopy $H\colon \AA^1_+ \otimes X_+  \rightarrow Y$ between maps $f, g\colon  X \rightarrow Y$. We obtain, using the right lax-structure, a homotopy
\[
\AA^1_+ \otimes \gamma_{C}(X) \rightarrow \gamma_{C}(\AA^1 \times X) \rightarrow \gamma_{C}(Y)
\]
between $\gamma_{C}(f)$ and $\gamma_{C}(g).$
\end{proof}

\begin{lemma} \label{prop:pres} The functor $\gamma_{C}\colon  \P_{\Sigma}(\Sm_{S+}) \rightarrow \P_{\Sigma}(\Sm_{S+})$ preserves $\LL_{\AA^1}$-equivalences. 
\end{lemma}

\begin{proof} By definition the class of $\LL_{\AA^1}$-equivalences is the strong saturation, in the sense of \cite{htt}*{Proposition 5.5.4.5}, of the maps in $\P_{\Sigma}(\Sm_{S+})$ by the (Yoneda image of) $\AA^1$-projections $\pi_X\colon  (\AA^1 \times X)_+ \simeq \AA^1_+ \otimes X_+ \rightarrow X_+$ for $X \in \Sm_S.$ According to \cite{bachmann-hoyois}*{Lemma 2.10} the class of $\LL_{\AA^1}$-equivalences is then generated under $2$-out-of-$3$ and sifted colimits by maps of the form $\pi_X \amalg \id_{Y_+}$ where $Y \in \Sm_S$. 

Since $\pi_X$ is an $\AA^1$-homotopy equivalence, it follows from Lemma~\ref{lem:homotopy-pres} that $\gamma_{C}(\pi_X)$ is an $\AA^1$-homotopy equivalence. Since $\gamma_{\C}$ preserves coproducts by assumption, the same is true for the morphism
\[
\gamma_{\C}(\pi_X \amalg \id_{Y_+}) \simeq \gamma_{C}(\pi_X) \amalg \gamma_{C}(\id_{Y_+}).
\] 
The functor $\gamma_{C}$ clearly preserves the $2$-out-of-$3$-property. Lastly, the functor $\gamma_{\C}$ preserves sifted colimits by definition and sifted colimits are computed valuewise in $\P_{\Sigma}(\Sm_{S+})^{\Delta^1}.$ Hence we conclude that $\gamma_{C}$ preserves $\LL_{\AA^1}$-equivalences. 
\end{proof}

\subsubsection{} 
Now we take into account a topology that we might want to put on $\Sm_{S+}$, namely, the topology of \emph{coproduct decomposition}. This is a topology on $\Sm_{S+}$ defined by a cd-structure, denoted by $\amalg,$ generated by squares
\[
\xymatrix{
S \ar[r] \ar[d] & U_+ \ar[d] \\
V_+ \ar[r] & X_+
}
\]
where $U$ and $V$ are summands of $X$ and $U \amalg V = X.$ Sheaves with respect to the topology generated by this cd-structure is exactly the nonabelian derived category on $\C.$ In other words we have
\[
\Shv_{\amalg}(\Sm_{S+}) \simeq \P_{\Sigma}(\Sm_{S+})
\]
by \cite{bachmann-hoyois}*{Lemma 2.4}. Hence all topologies $\tau$ considered in this paper satisfy $\Shv_{\tau}(\Sm_{S+}) \subseteq \P_{\Sigma}(\Sm_{S+}).$

\begin{definition} \label{tau-corr} Let $\tau$ be a topology on $\Sm_S$. Let $\C$ be a correspondence category with graph functor $\gamma_C\colon  \Sm_{S+} \rightarrow \C$. Then $\C$ \emph{is compatible with $\tau$} if for every $\tau$-sieve $U \hookrightarrow X$ in $\Sm_S$, the natural map 
\[
\gamma_\C(U_+) \rightarrow \gamma_\C(X_+)
\]
is an $\LL_{\tau}$-equivalence in $\P_{\Sigma}(\Sm_{S+})$.
\end{definition}

\begin{lemma} \label{lem:tau} 
Suppose that $\C$ is a correspondence category which is compatible with $\tau$. Then the functor 
\[
\gamma_{C}\colon  \P_{\Sigma}(\Sm_{S+}) \rightarrow \P_{\Sigma}(\Sm_{S+})
\] 
preserves $\LL_{\tau}$-equivalences.
\end{lemma}

\begin{proof}
By definition, the class of $\LL_{\tau}$-equivalences is the strong saturation, in the sense of \cite{htt}*{Proposition 5.5.4.5}, of the maps in $\P_{\Sigma}(\Sm_{S+})$ by the (Yoneda image of the) maps $i_+\colon  U_+ \hookrightarrow X_+$ where $X \in \Sm_S$ and $i$ is a $\tau$-sieve. According to \cite{bachmann-hoyois}*{Lemma 2.10}, the class of $\LL_{\tau}$-equivalences is then generated under $2$-out-of-$3$ and sifted colimits by maps of the form $\pi_X \amalg \id_{Y_+}$ for $Y \in \Sm_S$.  By the same reasoning as in Proposition~\ref{prop:pres} we need only check that $\gamma_\C(U_+) \rightarrow \gamma_\C(X_+)$ is an $\LL_{\tau}$-equivalence which is true by hypothesis. 
\end{proof}
From now on, whenever we consider a correspondence category $\C$ we make the following assumption on the topologies we discuss:
\begin{itemize}
\item The topology $\tau$ is at least as fine as the Nisnevich topology and is compatible in the sense of Definition~\ref{tau-corr}.
\end{itemize}

\subsubsection{} If $\C$ is a correspondence category, then we can construct its unstable motivic homotopy $\infty$-category in the usual way, as we now do. We consider two full subcategories of  $\P_{\Sigma}(\C)$ spanned by objects $\mathscr{F}$ satisfying the usual conditions:
\begin{itemize}
\item [($\P_{\AA^1}(\C)$)] The presheaf $\mathscr{F} \circ \gamma_C\colon  \Sm_S^{\op} \rightarrow \Spc$ is $\AA^1$-invariant. We denote the $\infty$-category spanned by such $\mathscr{F}$'s by $\P_{\AA^1}(\C).$
\item [($\Shv_{\tau}(\C)$)] The presheaf $\mathscr{F} \circ \gamma_C\colon  \Sm_S^{\op} \rightarrow \Spc$ is a $\tau$-sheaf. We denote the $\infty$-category spanned by such $\mathscr{F}$'s by $\Shv_{\tau}(\C).$
\end{itemize}

Since $\P_{\Sigma}(\C)$ is preadditive by Proposition~\eqref{prop:formal1}, we have a canonical equivalence $\CMon(\P_{\Sigma}(\C)) \simeq \P_{\Sigma}(\C).$ The $\infty$-category of \emph{unstable $\C$-motives}, denoted by $\H_{\tau}(\C)$, is then defined as $\P_{\AA^1}(\C) \cap \Shv_{\tau}(\C) \subseteq \P_{\Sigma}(\C)$. As usual we have localization functors $\LL_{\tau}^{\C}\colon  \P_{\Sigma}(\C) \rightarrow \Shv_{\tau}(\C),$ $\LL^{\C}_{\AA^1}\colon  \P_{\Sigma}(\C) \rightarrow \P_{\AA^1}(\C)$ and $\LL^{\C}_{\mot,\tau}\colon  \P_{\Sigma}(\C) \rightarrow \H_{\tau}(\C).$ From the construction of these localizations and the assumption on $\tau$, the adjunction~\eqref{gamma-adjunction} descends to an adjunction

\begin{equation} \label{gamma-adjunction}
\gamma^*_{C}: \H_{\tau}(\Sm_{S+}) \simeq \H_{\tau}(S)_{*} \rightleftarrows \H_{\tau}(\C): \gamma_{C*}
\end{equation}

\begin{lemma} \label{lem:preadd} 
The $\infty$-category $\H_{\tau}(\C)$ is preadditive. Hence we have a canonical equivalence $\CMon(\H_{\tau}(\C)^{\times}) \simeq \H_{\tau}(\C).$
\end{lemma}

\begin{proof} The $\infty$-category $\H_{\tau}(\C)$ is closed under finite products by checking that the conditions (Htpy) and ($\tau$-Desc) are preserved under taking products which are computed pointwise. The statement follows since $\P_{\Sigma}(\C)$ is preadditive by Proposition~\ref{prop:formal1}.
\end{proof}

\begin{definition}The $\infty$-category of \emph{effective $\C$-motives} $\H_{\tau}(\C)^{\mathrm{gp}}$ is defined to be the full subcategory of $\H_{\tau}(\C)$ spanned by the grouplike objects, in the sense of~\cite{ggn}*{Definition 1.2}.
\end{definition}

\subsubsection{} The next proposition captures the main property of categories of correspondences from the point of view of motivic homotopy theory.

\begin{proposition} \label{lem:compatible} Suppose that $\C$ is a correspondence category which is compatible with $\tau$. Then the functor
\[
\gamma_{C*}\colon \H_{\tau}(\C) \rightarrow \H_{\tau}(S)_*
\]
preserves sifted colimits and is conservative. Furthermore, $\H_{\tau}(\C)$ is canonically an $\H(S)_*$-module.
\end{proposition}

\begin{proof} For the first claim it suffices, after Proposition~\ref{prop:formal1}, to check that
\[\gamma_{\C*}\colon \P_{\Sigma}(\C) \rightarrow \P_{\Sigma}(\Sm_{S+}) \simeq \P_{\Sigma}(\Sm_S)_*
\] 
sends $\LL^{\C}_{\mot,\tau}$-equivalences to $\LL_{\mot,\tau}$-equivalences. This holds by Lemma~\ref{lem:compatible} and Lemma~\ref{prop:pres}. The assertion that $\H_{\tau}(\C)$ is an $\H(S)_*$-module follows from the right lax structure of $\gamma_{C*}$.
\end{proof}

\begin{remark} \label{rem:tau} If $\tau$ is a topology finer than the Nisnevich topology, then the fully faithful functor $\H_{\tau}(S)_* \rightarrow \H(S)_*$ need not preserve colimits. Hence the composite $\H_{\tau}(\C) \rightarrow \H_\tau(S)$ need not preserve colimits.
\end{remark}

From now on, if $\tau = \Nis$, we drop the decoration $\tau$ from $\H_{\tau}(\C)$ and $\H_{\tau}(S)_*$ and so forth.

\subsubsection{}From the above point of view, we see that $\gamma_{C*}$ is very close to preserving all colimits---we need only show that it preserves finite coproducts. The universal way to enforce this is to take commutative monoid objects on both sides with respect to Cartesian monoidal structures. We can do this for $\H_{\tau}(S)_*$ since it has finite products, and $\CMon(\H_{\tau}(\C)^{\times}) \simeq \H_{\tau}(\C)$ since it is preadditive \cite{ggn}*{Proposition 2.3}. We remark that the symmetric monoidal structure on $\P_{\Sigma}(\Sm_{S+})$ given by Day convolution is not Cartesian \footnote{Although the symmetric monoidal structure on $\P_{\Sigma}(\Sm_S)$ given by Day convolution is, and the natural sifted-colimit preserving functor $\P_{\Sigma}(\Sm_S) \rightarrow \P_{\Sigma}(\Sm_{S+})$ is symmetric monoidal.}.

To see this, consider the left adjoint to $\gamma_{C*}$, that is, 
\[
\gamma^*_C\colon  \H_{\tau}(S)_* \rightarrow \H_{\tau}(\C),
\] 
which preserves all small colimits. According to the universal property of $\CMon$ \cite{ggn}*{Corollary 4.9} we obtain an essentially unique functor $\gamma^*_\C\colon  \CMon(\H_{\tau}(S)^{\times}_*) \rightarrow \H_{\tau}(\C)$ since $\H_{\tau}(\C)$ is preadditive by Proposition~\ref{prop:formal1}.1. This functor admits a right adjoint $\gamma_{C*}\colon  \H_{\tau}(\C) \rightarrow \CMon(\H_{\tau}(S)_*^\times)$ which fits into a commutative diagram
\begin{equation} \label{factors}
\xymatrix{
 & \CMon(\H_{\tau}(S)_*^\times) \ar[d] \\
\H_{\tau}(\C) \ar[ur]^{\gamma^*_\C} \ar[r]^{\gamma_{C*}} & \H_{\tau}(S)_*. 
}
\end{equation}
In other words, the functor $\gamma_{C*}$ factors through the forgetful functor $\CMon(\H_{\tau}(S)_*^\times) \rightarrow \H_{\tau}(S)_*.$

\begin{proposition} \label{prop:cmon} Suppose that $\C$ is a correspondence category which is compatible with $\tau$. Then the functor
\[
\gamma_{C*}\colon \H_{\tau}(\C) \rightarrow \CMon(\H_{\tau}(S)^{\times}_*)
\]
preserves all small colimits and is conservative.
\end{proposition}

\begin{proof} By the diagram~\eqref{factors}, the functor $\gamma_{C*}$ preserves sifted colimits as the horizontal arrow preserves sifted colimits by Proposition~\ref{lem:compatible} and the vertical arrow preserves sifted colimits as a special case of~\cite{ggn}*{Proposition B.4}. Since it is a right adjoint it preserves finite products, but since its domain and codomain are preadditive it preserves finite coproducts as well and we are done by \cite{bachmann-hoyois}*{Lemma 2.8}. The conservativity statement follows from Proposition~\ref{lem:compatible} and the fact that the forgetful functor from commutative monoid objects is conservative.
\end{proof}

\subsubsection{$\mathbb{T}$-stability} Now we introduce a more refined notion than simply $\otimes$-inverting $\mathbb{T}$. This is inspired by the treatement of \cite{sag}*{Appendix C} on prestable $\infty$-categories.
\begin{definition} \label{def:t-prestable} Let $\C$ be an $\H(S)_*$-module in $\Cat_{\infty}.$ Then $\C$ is \emph{$\mathbb{T}$-prestable} if the endofunctor
\begin{equation} \label{tensor-by-t} 
  \mathbb{T} \otimes (-): \C \rightarrow \C
 \end{equation}
 is fully faithful. The $\infty$-category $\C$ is \emph{$\mathbb{T}$-stable} if the endofunctor~\eqref{tensor-by-t} is invertible.
\end{definition}

\begin{remark}The notion of a $\mathbb{T}$-stable $\infty$-category is a familiar one in motivic homotopy theory. In fact, $\mathbb{T}$-prestability is too---it is inspired by \emph{cancellation theorems} in the sense of \cite{voev-cancel} which assert that $\DM^{\eff}(k;\ZZ)$ is a $\mathbb{T}$-prestable for any perfect field $k$. The analogous statement holds for Milnor--Witt motivic cohomology as proved in \cite{cancellation}. The results of \cite{dk} give a framework for cancellation theorems. For the $\infty$-category of framed motivic spaces, cancellation holds as well  \cite{ehksy}*{Theorem 3.5.8}, which in turn relies on the cancellation theorem of Ananyevskiy, Garkusha and Panin \cite{agp}. Moreover, for any base scheme $S$, the subcategory $\SH(S)^{\eff} \subseteq \SH(S)$ of effective motivic spectra is $\mathbb{T}$-prestable.
\end{remark}

\subsubsection{} The thesis of Robalo \cite{robalo} provides a way to invert $\mathbb{T}$ for any $\H(C)_*$-module and obtain a symmetric monoidal stable $\infty$-category---in fact one that is a module over $\SH(S)$. We define the stable $\infty$-category of \emph{$\C$-motives} simply by
\[
\SH_{\tau}(\C) := \H_{\tau}(\C)[ \mathbb{T}^{\otimes-1}],
\] 
with notation as in \cite[Definition 2.6]{robalo}. We then have the basic adjunction
\[
\Sigma^{\infty}_{\mathbb{T}, \C}: \H_{\tau}(\C) \rightleftarrows \SH_{\tau}(\C): \Omega_{\mathbb{T},\C}^{\infty}.
\]
The following summarizes the basic properties of $\SH_{\tau}(\C)$:

\begin{proposition} \label{prop:sh-c} If $\C$ is correspondence category, then
\begin{enumerate}
\item The $\infty$-category $\SH_{\tau}(\C)$ is a stable presentably symmetric monoidal $\infty$-category, and
\item is generated under sifted colimits by objects of the form $\{ \mathbb{T}^{\otimes n} \otimes \Sigma^{\infty}_{\mathbb{T},\C}X \}_{n \in \ZZ, X\in \C}$.
\item The $\infty$-category $\SH_{\tau}(\C)$ is computed as the colimit in $\Mod_{\H(\Sm_S)_*}(\Pr^L)$ of
\begin{equation} \label{eq:colim}
\H_{\tau}(C) \xrightarrow{\mathbb{T}\otimes (-)} \H_{\tau}(\C) \xrightarrow{\mathbb{T}\otimes (-)}  \H_{\tau}(\C)\xrightarrow{\mathbb{T}\otimes (-)} \cdots.
\end{equation}
\item The functor
\[
\gamma_{C*}:\SH_{\tau}(\C) \rightarrow \SH_{\tau}(\Sm_S)
\]
is conservative and preserves colimits.
\end{enumerate}
\end{proposition}

\begin{proof} Stability follows from the standard equivalence $\mathbb{T} \simeq \GG_m \otimes S^1$ in $\SH(S)$, which remains true for modules over $\SH(S)$. The second assertion follows from the third via \cite{htt}*{Lemma 6.3.3.7} and the fact that $\H_{\tau}(\C)$ is generated under sifted colimits by representables which are smooth affine by the argument of \cite{adeel}*{Proposition 2.2.9} (which works for any topology $\tau$ finer than $\Nis$), while the third comes from \cite{robalo}*{Corollary 2.22}. The last assertion follows from Proposition~\ref{prop:cmon}.
\end{proof}

\subsubsection{}The last part of Proposition~\ref{prop:sh-c} is the main point of our axiomatization: the adjunction $\SH_{\tau}(S) \rightleftarrows \SH_{\tau}(\C)$ is monadic. In particular, if $\tau= \Nis$, then $\SH(S) \rightleftarrows \SH(\C)$ is monadic.

\subsubsection{From categories of correspondences to motivic module categories}\label{ad} 
Suppose that we have a functor 
\[
\C: \scrS^{\op} \rightarrow \CorrCat^{\otimes}
\]
which carries a morphism of schemes $f\colon T \rightarrow S$ to $f^*\colon \C_S \rightarrow \C_T$.
By naturality of the preceding constructions \footnote{The most nontrivial of which is the universal property of $\mathbb{T}$-stabilization for which we can appeal to \cite{bachmann-hoyois}*{Lemma 4.1}.}  we obtain a functor
\[
\SH_{\tau} \circ \C\colon \scrS^{\op} \rightarrow \Pr^{L,\otimes}_{\stab}
\]
equipped with a transformation $\SH|_{\scrS} \rightarrow \SH_{\tau} \circ \C$. We impose an additional assumption on $\C$, inspired by \cite{CDbook}*{Lemma 9.3.7}:
\begin{itemize}
\item For each $p\colon T \rightarrow S$, a smooth morphism in $\scrS$, the functor $p^*$ admits a left adjoint $p_{\#}$ such that the transformation $p_{\#}\gamma_{\C_T} \rightarrow \gamma_{\C_S}p_{\#}$ is an equivalence.
\end{itemize}
In this case, we say that $\C$ is \emph{adequate}. 

\subsubsection{} We employ the following additional notation: if $L\colon \SH(S) \rightarrow \SH(S)$ is a localization, denote by $L(\SH_{\tau}(\C_S))$ the subcategory of $\SH_{\tau}(\C_S)$ spanned by objects $X$ such that $\gamma_{C*}X$ is $L$-local. Since $\gamma_{C*}$ preserves limits, the inclusion $L(\SH_{\tau}(\C_S)) \hookrightarrow \SH_{\tau}(\C_S)$ is closed under limits and there is a localization functor (by the adjoint functor theorem) 
\[
L_{\C_S}\colon \SH_{\tau}(\C_S) \rightarrow L(\SH_{\tau}(\C_S))
\] 
rendering the following diagram commutative (since their right adjoints commute):
\[
\xymatrix{
\SH(S) \ar[d]_{L} \ar[r]^{\gamma_{\C_S}^*}  & \SH_{\tau}(\C_S) \ar[d]^{L_{\C_S}}\\
L(\SH(S)) \ar[r]^{\gamma_{\C_S}^*} & L(\SH_{\tau}(\C_S)).
}
\]
\begin{proposition} \label{prop:sh-c} If $\C\colon  \scrS^{\op} \rightarrow \CorrCat^{\otimes}$ is adequate, then the following hold:
\begin{enumerate}
\item We have premotivic adjunctions $\SH|_{\scrS} \rightleftarrows \SH_{\tau} \circ C$.
\item If $L$ is a smashing and symmetric monoidal localization of $\SH|_{\scrS}$, then we have a premotivic adjunction $L(\SH)|_{\scrS} \leftrightarrows L(\SH \circ \C)$.
\item If $\tau$ is a topology such that for each $S \in \scrS$, the functor $L(\SH_{\tau}(S)) \rightarrow L(\SH(S))$ preserves sifted colimits, then the premotivic adjunction $L(\SH)|_{\scrS} \rightleftarrows L(\SH_{\tau} \circ \C)$ is a motivic module category (in particular, this holds, when $\tau = \Nis$).
\end{enumerate}
\end{proposition}

\begin{proof} The proof of (1) follows as in the case of Grothendieck abelian categories \cite{CDbook}*{Corollary 10.3.11} and Voevodsky's $\C = \Cor$ (in the sense of \cite{CDbook}*{\S 9}) we give only the main points. Since $\C$ is adequate, we get that the equivalence $p_{\#}\gamma_{\C_T} \rightarrow \gamma_{\C_S}p_{\#}$ persists on the level of $\mathbb{T}$-stabilizations. What we need to verify, just as in \cite{CDbook}*{Proposition 10.3.9} is that the transformation $L_{\tau}\gamma_{C*} \simeq \gamma_{C*}L_{\tau}$ is an equivalence on the unstable level, i.e., the ``forgetful functor $\H_{\tau} \circ C \rightarrow \H|_{\scrS}$ preserves $\tau$-local objects and this is given by Lemma~\ref{lem:tau} under the standing assumption that $\C$ is compatible with $\tau$. The next two statements are then immediate from the definition of motivic module categories and the last statement of Proposition~\ref{prop:sh-c}.
\end{proof}

\subsection{Examples} \label{sec:ex} We now discuss examples of the above constructions and results.

\begin{example} \label{exmp0} Consider a Cartesian section of $\SH \rightarrow \scrS$, taking values in motivic $\mathcal{E}_{\infty}$-ring spectra. Then the transformation $\E \otimes(-)\colon  \SH|_{\scrS} \rightarrow \Mod_\E$ furnishes the first examples of motivic module categories. We can also consider further localizations of the premotivic category $\Mod_{\E}$, such as in \cite{elso} where $\scrS = \Sch_{\mathbb{Z}\left[\frac{1}{\ell}\right]}$ the localization functor is given by the composite of $\ell$-completion and \'etale localization, and $\E$ is $\MGL$; see \emph{loc. cit.} for more details where results in this paper is used to describe the $\infty$-category of modules over \'etale cobordism.
\end{example}

\begin{example} \label{exmp1:loc} Consider a localization $L\colon \SH|_{\scrS} \rightarrow L(\SH|_{\scrS})$. Then, if $L$ is smashing, $L(\SH|_{\scrS})$ is a motivic module category. Examples of these smashing localizations are given by \emph{periodization of elements}; we refer the reader to \cite{hoyois-cdh}*{Section 3} for an extensive discussion in our context. For example, a theorem of Bachmann \cite{ret} proves that periodizing the element $\rho$ yields real \'etale localization. If $x\colon \Sigma^{p,q}\sspt \rightarrow \sspt$, then the results of \cite{hoyois-cdh}*{\S 3} (or apply \cite{ret}*{Lemma 15}) tells us that $\sspt[x^{-1}]$ is an $\mathcal{E}_{\infty}$-ring and the projection formula holds. Thus, the category of $x$-periodic motivic spectra are modules over $\sspt[x^{-1}]$. 
\end{example}

\begin{example} \label{exmp2:tr} The basic example of a category of correspondences is Voevodsky's category of correspondences $\Cor_S$ in the sense of \cite{mvw}*{Appendix 1A} \cite{CDbook}*{\S 9};  this is defined for any Noetherian scheme $S$ \cite{CDbook}*{\S 9.1}. When $S$ is essentially smooth over a base field, the category Milnor--Witt correspondences $\widetilde{\Cor}_S$ of Calm\`es and Fasel \cite{calmes-fasel} is defined and is also a category of correspondences. Over a field (where both categories are defined), these categories are generalized by Garkusha's axioms in \cite{garkusha}. When defined, these categories are adequate in the sense of \S\ref{ad}. All these are examples of categories of correspondences, and thus gives rise motivic module categories. 
\end{example}

\begin{example}
Let $k$ be a perfect field. Given any $S\in\Sm_k$ and any good cohomology theory $A$ on $\Sm_S$ in the sense of \cite[§2]{dk}, then \cite[§3]{dk} defines an adequate category of correspondences $\Cor^A_S$ on $\Sm_S$. 
\end{example}

\begin{example} \label{exmp:fr} The $\infty$-category of framed correspondences of \cite{ehksy} is another example of a category of correspondences and is defined for any qcqs scheme $S$. The main theorem of \cite{hoyois-loc} asserts that the corresponding motivic module category is equivalent to $\SH(S)$, relying on the ``recognition principle" of \cite{ehksy}.
\end{example}

\begin{example}\label{ex:cor-e} If $\E \in \SH(S)$ is a homotopy associative ring spectrum, then \cite{ehksy2} defines a $\mathrm{h}\Spc$-enriched category $\mathrm{h}\Cor_S^\E$ of \emph{finite $\E$-correspondences}, which the authors expect to be the homotopy category of an $\infty$-category $\Cor_S^\E$ whenever $\E$ is actually an $\mathcal{A}_{\infty}$-ring. Setting $\C = \Cor_S^\E$, the $\infty$-category $\SH(\C)$ in this paper corresponds to $\DM^\E(S)$ in \emph{loc. cit}. We will discuss this example again in the next section.
\end{example}


\section{Module categories over regular schemes}\label{sec:regk}
 
In this section we show that the hypotheses of Theorem~\ref{thm:main} are satisfied for module categories over a field $k$, and more generally for module categories over regular $k$-schemes.

\subsection{The case of fields} We now aim to investigate when the hypotheses of Theorem~\ref{thm:main} are satisfied. One way to ensure that the projection formula holds is to use the following computation to reduce to the case of compact-rigid generation:

\begin{lemma} \label{dual} 
Suppose that we have an adjunction of symmetric monoidal $\infty$-categories 
\[
F: \C \rightleftarrows \D: G,
\] such that $F$ is strongly symmetric monoidal. Let $1\in\C$ denote the unit object of $\C$. If $\E\in\C$ is a strongly dualizable object, then the map $c\colon GF(1)\otimes\E\to GF(\E)$ is an equivalence.
 \end{lemma}
 
 \begin{proof} This follows from a standard computation: let $\E'\in\C$ be arbitrary, then we have a string of equivalences
\begin{align*}
\Maps_\C(\E',GF(1)\otimes\E)&\simeq \Maps_\C(\E'\otimes\E^\vee,GF(1))\\
&\simeq \Maps_\D(F(\E'\otimes\E^\vee),F(1))\\
&\simeq \Maps_\D(F(\E')\otimes F(\E)^\vee,F(1))\\
&\simeq \Maps_\D(F(\E'),F(\E))\\
&\simeq \Maps_\C(\E',GF(\E)),
\end{align*}
which shows the claim.
\end{proof}

\subsubsection{}Thus, if $\SH(S)$ is generated by strongly dualizable objects, it follows that the projection formula holds:
 
\begin{theorem} \label{thm:fieldcase}
Let $k$ be a field. Suppose that $\ell$ is a prime which is coprime to the exponential characteristic $e$ of $k$ and $\scrM$ is a motivic module category on $k$. Then we have the following equivalences of presentably symmetric monoidal stable $\infty$-categories:
\[
L_{(\ell)}(\scrM(k))\simeq\Mod_{L_{(\ell)}\gamma_*\gamma^*(\sspt_S)}(\SH(k)),\]
and
\[
\scrM(k)\left[\tfrac{1}{e}\right] \simeq\Mod_{\gamma_*\gamma^*(\sspt)\left[\frac{1}{e}\right]}(\SH(k)).
\]
\end{theorem}

\begin{proof} After Theorem~\ref{thm:main}, we need to verify the appropriate projection formulas. By assumption, the functor $\gamma_*$ preserves sifted colimits and thus the functors $\gamma_*\gamma^*(\sspt_S) \otimes (-)$ and $\gamma_*\gamma^*(-)$ do as well. Now Lemma~\ref{dual} tells us that the projection formula holds for strongly dualizable objects in $\SH(k)_{(\ell)}$. Thus we will be done if we can prove that the second inclusion of~\eqref{riou-inclusions}, $\SH^{\rig}(k)_{(\ell)} \subseteq \SH^{\omega}(k)_{(\ell)}$, is an equality. This amounts to showing that $\SH(k)_{(\ell)}$ is in fact generated by sifted colimits by strongly dualizable objects. This follows by Example~\ref{exmp-field}, which also verifies the theorem for the $e$-inverted case.
\end{proof}

\subsubsection{} We now obtain the following extension of \cite{ro}*{Theorem 1}, \cite{hko}*{Theorem 5.8}, \cite{garkusha}*{Theorem 5.3}, \cite{bachmann-fasel}*{Lemma 5.3}:

\begin{corollary} \label{thm:mz-mz-tild} Let $k$ be a field of exponential characteristic $e$ and let $\gamma_{\C}\colon \Sm_k \rightarrow \C$ be a correspondence category. Then there is an equivalence of presentably symmetric monoidal stable $\infty$-categories
\[
\SH(\C)\left[\tfrac{1}{e}\right] \simeq\Mod_{\gamma_{C*}\gamma^*_C(\sspt)\left[\tfrac{1}{e}\right]}(\SH(k)).
\]
\end{corollary}

\subsection{The case $\scrS=\Reg_k$}
Following \cite{equichar}, we can extend the previous result to the category $\Reg_k$ of finite dimensional Noetherian schemes that are regular over a field, provided that we impose some additional assumptions on $\scrM$. For the rest of this section, we will therefore assume that $\scrM$ is a motivic module category which in addition satisfies the following property:

\begin{itemize}
\item The premotivic category $\scrM$ satisfies localization (Definition~\ref{def:loci}) and continuity (Definition~\ref{def:cont-loc}).
\end{itemize}

\begin{lemma} \label{lem:t-vs-pull} Suppose that $f\colon T \rightarrow S$ is a morphism in $\Reg_k$. In the following cases, the transformation
\[
f^*\gamma_* \rightarrow \gamma_*f^*
\]
is an equivalence:
\begin{enumerate}
\item The map $f$ is an inverse limit
\[
f = \varprojlim_\alpha f_{\alpha} T_{\alpha} \rightarrow S,
\]
where the transition maps $f_{\alpha \beta}\colon T_{\alpha} \rightarrow T_{\beta}$ are dominant, affine and smooth.
\item The map $f$ is a closed immersion and
\[
S \simeq \varprojlim_\alpha S_{\alpha},
\]
where each $S_{\alpha}$ is a smooth, separated $k$-scheme of finite type with flat affine transition maps.
\end{enumerate}
\end{lemma}

\begin{proof} 
Under the continuity and localization assumption on $\scrM$, the proof in \cite{equichar}*{Lemma 3.20} for the case of $\scrM = \DM(-,R)$ applies verbatim.
\end{proof}

\subsubsection{} We now have the following extension of Theorem~\ref{thm:fieldcase}.

\begin{theorem} \label{thm:reg-mod} Let $k$ be a field of exponential characteristic $e$, and let $\scrM$ be a motivic module category on $\Reg_k$. Then the functor $\gamma^*\colon \SH \rightarrow \scrM$ induces a canonical equivalence
\[
\Mod_{\gamma_*\gamma^*(\sspt)\left[\tfrac{1}{e}\right]}(\SH(-)) \xrightarrow{\sim} \scrM\left[\tfrac{1}{e}\right]
\]
of premotivic categories on $\Reg_k$.
\end{theorem}

\begin{proof} 
After Theorem~\ref{thm:main}, our goal is to verify that $(\SH|_{\Reg_k},\scrM)$ satisfies the projection formula. Suppose that $S \in \Reg_k$, and  let $\E \in \SH(S)$. We claim that the map
\begin{equation} \label{eq:map-e}
\gamma_*\gamma^*(\sspt_S) \otimes \E \rightarrow \gamma_*\gamma^*(\E)
\end{equation}
is an equivalence. To show this, we follow closely the logic of \cite{equichar}*{Theorem 3.1}. 

First, assume that $S$ is an essentially smooth scheme over a field. For each $x \in S$, we write $S_x$ for the localization of $S$ at $x$. Then the family of functors
\[
\{ \SH(S) \rightarrow \SH(S_x) \}
\]
is conservative by \cite{CDbook}*{Proposition 4.3.9}. Hence we are reduced to proving that the map~\eqref{eq:map-e} is an equivalence in the case that $S$ is furthermore \emph{local}. In this case, let $i\colon x \hookrightarrow S_x$ be the closed point and write $j\colon U_x \rightarrow S_x$ for the open complement. By our assumption on $S$, $U_x$ has dimension $< \dim\,S$. We consider the following commutative diagram, where the rows are cofiber sequences:
\begin{equation} \label{eq:open-cls}
\xymatrix{
j_!(j^*\gamma_*\gamma^*(\sspt_S) \otimes j^*\E)  \ar[r] \ar[d] & \gamma_*\gamma^*(\sspt_S) \otimes \E \ar[r] \ar[d] &  i_*(i^*\gamma_*\gamma^*(\sspt_S) \otimes i^*\E) \ar[d]  \\
j_!j^* \gamma_*\gamma^*(\E) \ar[r] \ar[d]_{f_1} & \gamma_*\gamma^*(\E) \ar[d]_{=} \ar[r]  & i_*i^*\gamma_*\gamma^*(\E) \ar[d]^{f_2} \\
j_!\gamma_*\gamma^*j^*\E \ar[r] & \gamma_*\gamma^*\E \ar[r] & i_*\gamma_*\gamma^*i^*\E.
}
\end{equation}
Now,
\begin{itemize}
\item The left vertical composite is an equivalence because (1) $j_*$ commutes with $\gamma_*$ by definition of a morphism of premotivic categories, 
and (2) by the induction hypothesis.
\item The right vertical composite is an equivalence using (1) Lemma~\ref{lem:t-vs-pull}.2 and (2) the case of fields, Theorem~\ref{thm:fieldcase}.
\end{itemize}
It therefore remains to show that $f_1$ and $f_2$ are equivalences.
\begin{itemize}
\item The map $f_1$ is an equivalence because $j_*$ commutes with $\gamma_*$. 
\item That $f_2$ is an equivalence follows from Lemma~\ref{lem:t-vs-pull}.2. 
\end{itemize}

Now, following the ``\emph{General case}'' of \cite{equichar}, we explain how the bootstrap to regular $k$-schemes work. By continuity (appealing to \cite{CDbook}*{Proposition 4.3.9} again), we may again assume that $S$ is a \emph{Henselian local} regular $k$-scheme. As explained in \emph{loc. cit.}, there is a sequence of regular Noetherian $k$-schemes
\[
T \stackrel{f}{\rightarrow} S' \stackrel{q}{\rightarrow} S
\]
such that the following hold:
\begin{itemize}
\item The scheme $S'$ has infinite residue field and the functor $q^*\colon \SH(S)\left[\frac{1}{e}\right] \rightarrow \SH(S')\left[\frac{1}{e}\right]$ is conservative.
\item The scheme $T$ is the \emph{$\infty$-gonflement} of $\Gamma(S', \Oscr_{S'})$ \cite{equichar}*{Definition 3.21} and the functor $f^*\colon \SH(S')\left[\frac{1}{e}\right] \rightarrow \SH(T)\left[\frac{1}{e}\right]$ is conservative.
\item Both $f$ and $q$ satisfy the hypotheses of Lemma~\ref{lem:t-vs-pull}.1, and thus $f^*$ and $q^*$ commute with $\gamma_*$.
\end{itemize}
Hence, to check that the map~\eqref{eq:map-e} is an equivalence it suffices to check that it is an equivalence after applying $(qf)^*$. Since $T$ is, by construction, the spectrum of a filtered union of its smooth subalgebras we invoke continuity of $\SH$ to conclude.
\end{proof}

\newcommand{\BM}{\mathrm{BM}}
\subsubsection{} Lastly, we have the following class of examples of motivic module categories for which localization and continuity holds. We will make the following assumption:
\begin{itemize}
\item for a base scheme $S$ and $\mathcal{A}_\infty$-ring spectrum $\E \in \SH(S)$, there exists an $\infty$-category $\Cor_S^\E$ such that its homotopy category is the $\mathrm{h}\Spc$-enriched category $\mathrm{h}\Cor_S^\E$ of \cite{ehksy2}. 
\end{itemize}
With this assumption in play that any motivic $\mathcal{A}_\infty$-ring spectrum $\E$ gives rise the motivic module category $\DM^\E$ \cite{ehksy2} as explained in Example~\ref{ex:cor-e}. While this makes the next results conditional, we will explain unconditional instances of these results in Example~\ref{ex:hz}.

\begin{proposition} \label{prop:e-cts} Let $\scrS \subseteq \Sch_S$. Then, for any $\mathcal{A}_\infty$-ring spectrum $\E \in \SH(S)$, premotivic category $\DM^{\E}\colon \scrS^{\op} \rightarrow \Cat_{\infty}$ satisfies continuity for dominant affine morphisms.
\end{proposition}

\begin{proof} We first claim the analog of \cite{CDbook}*{Proposition 9.3.9} for $\E$-correspondences. Let $(X_i)_{i \in I}$ be a cofiltered diagram of separated $S$-schemes of finite type with affine dominant transition morphisms. Let $X = \varprojlim_i X_i$, which is assumed to exist in $\Sch_S$ and is assumed to be Noetherian. Then we claim that for any separated $S$-scheme $Y$ of finite type, the map
\begin{equation} \label{eq:compare}
\colim_{i \in I^{\op}} \Cor_S^{\E}(X_i, Y) \rightarrow \Cor_S^{\E}(X, Y)
\end{equation}
is an equivalence. 

To do so, we use the dual of \cite{ehksy}*{Lemma 4.1.26}. Denote by $c_{X_i}$ (resp. $c_{X}$) be the filtered poset of reduced subschemes of $X_i \times_S Y$ (resp. $X \times_S Y$) which are finite and universally open over $X_i$ (resp. $X$). We denote by $\mathrm{Sub}(c_X)$ the poset of full sub-posets of $c_X$. We then have a functor $K\colon I \rightarrow \mathrm{Sub}(c_X)$, $i \mapsto K_i = c_{X_i}$ where $c_{X_i}$ is regarded as a full sub-poset in the obvious way. By continuity of $\SH$, the functor $\E^{\BM}(-/X)\colon c_X \rightarrow \Spc$ restricts to a functor $\E^{\BM}(-/X_i)\colon c_{X_i} \rightarrow \Spc$. Hence the map~\eqref{eq:compare} is, by \cite{ehksy2}*{Definition 4.1.1}, equal to the map
\[
\colim_{I^{\op}} \colim_{c_{X_i}} \E^{\BM}(Z_i/X_i) \rightarrow  \colim_{Z \in c_X} \E^{\BM}(Z/X),
\]
which we claim is an equivalence. The hypotheses of  \cite{ehksy}*{Lemma 4.1.26} follow easily (under the hypotheses that the transition maps are affine and dominant) by \cite{CDbook}*{Proposition 8.3.9, 8.3.6}. Hence the desired claim follows. The rest of the proof follows as in the case of $\DM$ from \cite{CDbook}*{Theorem 11.1.24}.\end{proof}

\begin{proposition} \label{prop:e-cts} Let $k$ be a field and let $\E \in \SH(k)$ be a $\mathcal{A}_\infty$-ring spectrum. Then the premotivic category $\DM^{\E}\colon \scrS^{\op} \rightarrow \Cat_{\infty}$ satisfies $\Loc_i$ whenever $i$ is a closed immersion of regular schemes.
\end{proposition}

\begin{proof}
Since $\DM^\E$ is constructed from Nisnevich local objects, it is Nisnevich separated. By \cite{CDbook}*{Proposition 6.3.14}, it has the weak localization property, i.e., it has $\Loc_i$ for any closed immersion with smooth retractions. Arguing as in \cite{CDbook}*{Corollary 6.3.15}, it has the localization property with respect to any closed immersion between smooth schemes. The rest of the argument then follows as in \cite{equichar}*{Proposition 3.12}, which uses the continuity results established in Proposition~\ref{prop:e-cts} as above.
\end{proof}

\subsubsection{} From this we conclude:

\begin{corollary} \label{cor:reg} Let $k$ be a field and $\E \in \SH(k)$ an $\mathcal{A}_{\infty}$ ring spectrum. Then we have a canonical equivalence
\[
\DM^\E\left[\tfrac1e\right]\simeq \Mod_{\gamma_*\gamma^*(\sspt)\left[\tfrac1e\right]}(\SH(-))
\]
of premotivic categories on on $\Reg_k$.
\end{corollary}

\begin{example} \label{ex:hz} As explained in \cite{ehksy2}*{4.1.19}, the hypothetical $\infty$-category $\Cor^\E_S$ is equvialent to $\mathrm{h}\Cor^\E_S$ whenever $S$ is a essentially smooth over a perfect field $k$ and $\E$ is pulled back from $k$ and lies in the heart of the effective homotopy $t$-structure over $k$. Hence, Theorem~\ref{cor:reg} holds unconditionally whenever $\E$ is pulled back from the prime subfield of $k$ and lies in the heart of the effective homotopy $t$-structure there. 

Examples of such spectra are the motivic cohomology spectrum $\H\ZZ$ and its Milnor-Witt counterpart $\H\widetilde\ZZ$. Furthermore \cite{ehksy2}*{Proposition 4.3.6} (resp. \cite{ehksy2}*{Proposition 4.3.19}), it is proved that $\DM^{\H\ZZ}(S) \simeq \DM(S)$ (resp. $\DM^{\H\widetilde\ZZ}(S) \simeq \widetilde{\DM}(S)$) whenever $S$ is essentially smooth over a Dedekind domain (resp. essentially smooth over a perfect field) \cite{ehksy2}*{Proposition 4.3.8} (resp. \cite{ehksy2}*{Proposition 4.3.19}). By the continuity result of Proposition~\ref{prop:e-cts} we can enhance the comparison results for $\DM$ to regular schemes over fields. 
While $\widetilde{\DM}(S)$ is not defined outside of smooth schemes over fields, Corollary~\ref{cor:reg} promotes the comparison results between $\widetilde{\DM}$ and modules over $\H\widetilde{\ZZ}$ of \cite{garkusha} and \cite{bachmann-fasel} at least to smooth schemes over fields. We contend, however, that $\DM^{\H\widetilde\ZZ}(S)$ is a decent definition for $\widetilde{\DM}(S)$ in general.
\end{example}

\end{document}